\documentclass[12pt,a4paper]{amsart}
\usepackage[top=30mm, bottom=30mm, left=30mm, right=30mm]{geometry}
\usepackage{graphicx}
\usepackage{amsthm}
\usepackage{color,xcolor}
\usepackage[colorlinks=true,citecolor=blue]{hyperref}
\usepackage[all,cmtip]{xy}
\newtheorem{theorem}{Theorem}[section]
\newtheorem{lemma}[theorem]{Lemma}
\theoremstyle{definition}
\newtheorem{definition}[theorem]{Definition}

\newtheorem{cor}[theorem]{Corollary}
\newtheorem{remark}[theorem]{Remark}
\newtheorem{ques}[theorem]{Question}

\newtheorem{prop}[theorem]{Proposition}
\def \RP {{\bf RP}}
\def \Z  {\mathbb{Z}}
\def \N  {\mathbb{N}}
\def \Q {{\bf Q}}
\def \I {\mathrm{Ind}}
\begin{document}
\title
{Independence and almost automorphy of higher order}
\author[J.~Qiu]{Jiahao Qiu}
\address[J.~Qiu]{Wu Wen-Tsun Key Laboratory of Mathematics, USTC, Chinese Academy of Sciences and
School of Mathematics, University of Science and Technology of China,
Hefei, Anhui, 230026, P.R. China}
\email{qiujh@mail.ustc.edu.cn}

\begin{abstract}
In this paper, it is shown that for a minimal system $(X,T)$ and $d,k\in \N$,
if $(x,x_i)$ is regionally proximal of order $d$ for $1\leq i\leq k$,
then $(x,x_1,\ldots,x_k)$ is $(k+1)$-regionally proximal of order $d$.

Meanwhile,
we introduce
the notion of $\mathrm{IN}^{[d]}$-pair:
for a dynamical system $(X,T)$ and $d\in \N$,
a pair $(x_0,x_1)\in X\times X$ is called an $\mathrm{IN}^{[d]}$-pair
if for any $k\in \N$ and any neighborhoods $U_0 ,U_1 $ of $x_0$ and $x_1$ respectively,
there exist integers $p_j^{(i)},1\leq i\leq k,$ $1\leq j\leq d$ such that
$$
\bigcup_{i=1}^k\{ p_1^{(i)}\epsilon(1)+\ldots+p_d^{(i)}
\epsilon(d):\epsilon(j)\in \{0,1\},1\leq j\leq d\}\backslash \{0\}\subset \I(U_0,U_1),
$$
where $\I(U_0,U_1)$ denotes the collection of all independence sets for $(U_0,U_1)$.
It turns out that
for a minimal system,
if it does not contain any nontrivial $\mathrm{IN}^{[d]}$-pair,
then it is an almost one-to-one extension of
its maximal factor of order $d$.
\end{abstract}

\date{\today}

\maketitle

\section{Introduction}

By a \emph{topological dynamical system} or just a \emph{dynamical system},
we mean a pair $(X,T)$, where $X$ is a compact metric space with a metric $\rho$
and $T:X\to X$ is a homeomorphism.

\medskip

In the recent years, the study of the dynamics of rotations on nilmanifolds
and inverse limits of this kind of dynamics has drawn much interest, since it relates
to many dynamical properties and has important applications in number theory.
We refer to \cite{HK18} and the references therein for a systematic treatment on the subject.

\medskip

In a pioneer work, Host-Kra-Maass in \cite{HKM} introduced the notion of
{\it regionally proximal relation of order $d$}
for a dynamical system $(X,T)$, denoted by $\mathbf{RP}^{[d]}(X)$.
For $d\in\N$, we say that a minimal system is a \emph{system of order d}
if $\mathbf{RP}^{[d]}(X)=\Delta$ and this is equivalent for $(X,T)$ to be
an inverse limit of nilrotations on $d$-step nilsystems (see \cite[Theorem 2.8]{HKM}).
For a minimal distal system $(X,T)$, it was proved that
$\mathbf{RP}^{[d]}(X)$ is an equivalence relation and $X/\mathbf{RP}^{[d]}(X)$
is the maximal factor of order $d$ \cite{HKM}.
Then Shao-Ye \cite{SY12} showed that in fact for
any minimal system, $\mathbf{RP}^{[d]}(X)$ is an equivalence
relation and $\mathbf{RP}^{[d]}(X)$
has the so-called lifting property.

In \cite{HLY11}, the notion of \emph{$k$-regional proximal relation} was introduced.
It was shown that for a minimal system $(X,T)$ and $k\geq 2$,
if $(x,x_{i})$ is regionally proximal for all $1\leq i\leq k$,
then $(x_1,\ldots,x_k)$ is $k$-regionally proximal,
i.e. for every $\delta>0$, there exist $x_i'\in X,1\leq i\leq k$ and $n\in \Z$
  such that $\rho(x_i,x_i')<\delta$ and
  $ \rho(T^n x_1',T^nx_i')<\delta,1\leq i\leq k$.
In this paper, we extend this result to higher order (Theorem \ref{main-thm1}).

\medskip

Following the \emph{local entropy theory}, for a survey see \cite{GY09},
each dynamical system admits a
maximal zero topological entropy factor, and this factor is induced by the smallest
closed invariant equivalence relation containing \emph{entropy pairs} \cite{BL93}.
In \cite{HY06}, entropy pairs are characterized as those pairs that
admit an \emph{interpolating set} of positive density.
Later on, the notions of \emph{sequence entropy pairs} \cite{HLSY03} and \emph{untame pairs}
(called \emph{scrambled pairs} in \cite{WH06}) were introduced.
In \cite{KL07} the concept of \emph{independence}
was extensively studied and used to unify the aforementioned notions.
Let $(X,T)$ be a dynamical system and $\mathcal{A}=(U_0,U_1,\ldots,U_k)$ be a tuple
of subsets of $X$. We say that a subset $F\subset \Z$ is an \emph{independence set} for $\mathcal{A}$
if for any nonempty finite subset $J\subset F$ and any $s=(s(j):j\in J)\in \{0,1,\ldots,k\}^J$
we have $\bigcap_{j\in J}T^{-j}U_{s(j)}\neq \emptyset$.
It was shown that a pair of points $x_0,x_1$ in $X$ is a sequence entropy pair if and only
if each $\mathcal{A}=(U_0,U_1)$, where $U_0$ and $U_1$ are
neighborhoods of $x_0$ and $x_1$ respectively, has arbitrarily long finite
independence sets. Also, the pair is an untame pair if and only if each $\mathcal{A} = (U_0 ,U_1 )$
as before has infinite independence sets.
It was shown (\cite{EG07,HLSY03,KL07}) that a
minimal null (resp. tame) system is an almost one-to-one extension of its maximal equicontinuous
factor and is uniquely ergodic.

For $d\in \N$ and $p_1,\ldots,p_d\in \Z$,
we call the set $\{p_1\epsilon(1)+\ldots+p_d\epsilon(d):\epsilon(j)\in \{0,1\},1\leq j\leq d\}\backslash\{0\}$
an \emph{$IP_d$-set}.
The notion of \emph{$Ind_{fip}$-pair} was studied in \cite{DDMSY13}:
a pair of points $x_0,x_1$ in $X$ is an $\text{Ind}_{fip}$-pair if and only
if the independence sets for each $\mathcal{A}=(U_0,U_1)$ as before
contain an $\mathrm{IP}_d$-set for any $d\in \N$.
It was showed that
a minimal system without any nontrivial $\mathrm{Ind}_{fip}$-pair
is an almost one-to-one extension of
its maximal factor of order $\infty$.

\medskip

So, it is natural to ask: can we give a finer classification of almost automorphy of higher order
using independence?

In this paper, we introduce the notion of \emph{$IN^{[d]}$-pair}:
a pair of points $x_0,x_1$ in $X$ is an $\mathrm{IN}^{[d]}$-pair
if and only if
the independence sets for each $\mathcal{A} = (U_0 ,U_1 )$ as before
contain a union of arbitrarily finitely many $\mathrm{IP}_d$-sets.
Using dynamical cubespaces,
we first provide a characterization of $\mathrm{IN}^{[d]}$-pairs for minimal systems
(Lemma \ref{INN}).
By \cite[Chapter 6]{HK18}, the dynamical cubespaces of minimal nilsystems can also be viewed as nilsystems.
Following this,
it is shown that for minimal nilsystems,
nontrivial regionally proximal of order $d$ pairs
are $\mathrm{IN}^{[d]}$-pairs
 (Theorem \ref{ind-nilsystem}).
 Moreover, this property also holds for inverse limits of minimal nilsystems.

 For a minimal system and $d\in \N$,
 by reducing on the maximal factor of order $\infty$
 which is an inverse limit of minimal nilsystems \cite{DDMSY13},
 we can show that any
 nontrivial regionally proximal of order $d$ pair
is an $\mathrm{IN}^{[d]}$-pair if it is minimal in the product system (Lemma \ref{main-thm3}).
Among other things it turns out that for
a minimal system
if it does not contain any nontrivial $\mathrm{IN}^{[d]}$-pair,
then it is an almost one-to-one extension of
its maximal factor of order $d$ (Theorem \ref{main-thm2}).

\medskip

The paper is organized as follows.
In Section 2, the basic notions used in the paper are introduced.
In Section 3, we discuss the $k$-regionally proximal relation of higher order (Theorem \ref{main-thm1}).
In Section 4, it is shown that
for a minimal nilsystem
any regionally proximal of order $d$ pair is an $\mathrm{IN}^{[d]}$-pair
(Theorem \ref{ind-nilsystem}).
In the final section,
among other things we show that
for a minimal system
if it does not contain any nontrivial $\mathrm{IN}^{[d]}$-pair,
then it is an almost one-to-one extension of
its maximal factor of order $d$ (Theorem \ref{main-thm2}).

\medskip

\noindent {\bf Acknowledgments.}
The author would like to thank Professor X. Ye
for helping discussions and remarks.
The author was supported by NNSF of China (11431012).

\section{Preliminaries}
In this section we gather definitions and preliminary results that
will be necessary later on.
Let $\N$ and $\Z$ be the sets of all positive integers
and integers respectively.

\subsection{Topological dynamical systems}

 A \emph{topological dynamical system}
 (or \emph{dynamical system}) is a pair $(X,T)$,
 where $X$ is a compact metric space with a metric $\rho$ and $T : X \to  X$
is a homeomorphism.
If $A$ is a non-empty closed subset of $X$ and $TA\subset A$, then $(A,T|_A)$ is called a \emph{subsystem} of $(X,T)$,
where $T|_A$ is the restriction of $T$ on $A$.
If there is no ambiguity, we use the notation $T$ instead of $T|_A$.
For $x\in X,\mathcal{O}(x,T)=\{T^nx: n\in \Z\}$ denotes the \emph{orbit} of $x$.
A dynamical system $(X,T)$ is called \emph{minimal} if
every point has dense orbit in $X$.
A subset $Y$ of $X$ is called \emph{minimal} if $(Y,T)$ is a minimal subsystem of $(X,T)$.
A point $x\in X$ is called \emph{minimal} if it is contained in a minimal set $Y$ or,
equivalently, if the subsystem $(\overline{\mathcal{O}(x,T)},T)$ is minimal.

A \emph{homomorphism} between the dynamical systems $(X,T)$ and $(Y,T)$ is a continuous onto map
$\pi:X\to Y$ which intertwines the actions; one says that $(Y,T)$ is a \emph{factor} of $(X,T)$
and that $(X,T)$ is an \emph{extension} of $(Y,T)$. One also refers to $\pi$ as a \emph{factor map} or
an \emph{extension} and one uses the notation $\pi : (X,T) \to (Y,T)$. The systems are said
to be \emph{conjugate} if $\pi$ is a bijection. An extension $\pi$ is determined
by the corresponding closed invariant equivalence relation $R_\pi=\{(x,x')\in X\times X\colon \pi(x)=\pi(x')\}$.
An extension $\pi : (X,T) \to (Y,T)$
is \emph{almost one-to-one} if the $G_\delta$ set $X_0=\{x\in X:\pi^{-1}(\pi (x))=\{x\}\}$ is dense.

\subsection{Discrete cubes and faces}
Let $X$ be a set and let $d\ge 1$ be an integer. We view the element $\epsilon\in\{0, 1\}^d$ as
a sequence $\epsilon=(\epsilon(1),\ldots, \epsilon(d))$, where $\epsilon(i)\in\{0,1\},$
$1\leq i \leq d$.
If $\vec{n} = (n_1,\ldots, n_d)\in \Z^d$ and $\epsilon\in \{0,1\}^d$, we
define
\[
 \vec{n}\cdot \epsilon = \sum_{i=1}^d n_i\epsilon(i) .
 \]

We denote the set of maps $\{0,1\}^{d}\to X$
by $X^{[d]}$.
For $\epsilon\in \{0,1\}^d$ and $\mathbf{x}\in X^{[d]}$,
$\mathbf{x}(\epsilon)$ will be used to denote the $\epsilon$-component of $\mathbf{x}$.
For $x\in X$, write $x^{[d]}=(x,x,\ldots,x)\in X^{[d]}$.
The \emph{diagonal} of $X^{[d]}$ is $\Delta^{[d]}=\Delta^{[d]}(X)=\{ x^{[d]}:x\in X\}$.
Usually, when $d=1$, we denote the diagonal by $\Delta_X$ or $\Delta$
instead of $\Delta^{[1]}$.
We can isolate the first coordinate,
writing $X^{[d]}_*=X^{2^d-1}$ and writing $\mathbf{x}\in X^{[d]}$
as $\mathbf{x}=(\mathbf{x}(\vec{0}),\mathbf{x}_*)$,
where $\mathbf{x}_*=(\mathbf{x}(\epsilon):\epsilon\in \{0,1\}^d\backslash \{ \vec{0}\})\in X^{[d]}_*$.

Identifying $\{0,1\}^d$ with the set of vertices of the Euclidean unit cube,
an Euclidean isometry of the unit cube permutes the vertices of the
cube and thus the coordinates
of a point $\mathbf{x}\in X^{[d]}$.
These permutations are the \emph{Euclidean permutations} of $X^{[d]}$.

\medskip

A set of the form
\begin{equation}\label{face-form}
F=\{\epsilon\in\{0,1\}^{d}: \epsilon(i_1)=a_{1},\ldots,\epsilon(i_{k})=a_{k}\}
\end{equation}
for some $k\geq0$, $1\leq i_{1}<\ldots<i_{k}\leq d$ and
$a_{i}\in\{0,1\}$ is called a \emph{face} of \emph{codimension} $k$ of the discrete
cube $\{0,1\}^{d}$.\footnote{The case $k=0$ corresponds to $\{0,1\}^{d}$.}
A face of codimension $1$ is called a \emph{hyperface}.
If all $a_{i}=1$ we say that
the face is \emph{upper}. Note all upper faces contain $\vec{1}$
and there are exactly $2^d$ upper faces.

For $\epsilon,\epsilon'\in \{0,1\}^d$, we say that $\epsilon\geq \epsilon'$ if
$\epsilon(i)\geq \epsilon'(i)$ for all $1\leq i\leq d$.
Let $F$ be a face of $\{0,1\}^d$,
the \emph{smallest element} of the face $F$ is defined by $\min F$,
meaning that $\min F\in F$
and $\epsilon\geq   \min F$ for all $\epsilon\in F$.
Indeed, if a face $F$ has form (\ref{face-form}),
then $\min F(i_j)=a_j$ for $1\leq j\leq k$,
and $\min F(i)=0$ for $i\in \{1,\ldots,d\}\backslash \{i_1,\ldots,i_k\}$.

\subsection{Dynamical cubespaces}
Let $(X,T)$ be a dynamical system and $d\in \N$.
We define $\Q^{[d]}(X)$ to be the closure in $X^{[d]}$ of elements of the form
\[
(T^{\vec{n}\cdot\epsilon}x=T^{n_1\epsilon(1)+\ldots+n_d\epsilon(d)}x:\epsilon\in\{0,1\}^d),
\]
where $\vec{n}=(n_1,\ldots,n_d)\in \Z^d$ and $x\in X$.
We call this set the \emph{dynamical cubespace of dimension d} of the system.

It is important to note that $\mathbf{Q}^{[d]}(X)$ is invariant under the Euclidean permutations of $X^{[d]}$.

\begin{definition}
{\em Face transformations} are defined inductively as follows: Let
$T^{[1]}_1=\mathrm{id} \times T$. If
$\{T^{[d-1]}_j\}_{j=1}^{d-1}$ is defined already, then set
$$T^{[d]}_j=T^{[d-1]}_j\times T^{[d-1]}_j, 1\leq j\leq d-1,$$
$$T^{[d]}_d=\mathrm{id} ^{[d-1]}\times T^{[d-1]}.$$
\end{definition}

It is easy to see that for $1\leq j\leq d$, the face transformation
$T^{[d]}_j : X^{[d]}\rightarrow X^{[d]}$ can be defined by, for
every ${\bf x} \in X^{[d]}$ and $\epsilon\in \{0,1\}^d $,
\[
T^{[d]}_j{\bf x}=
\left\{
  \begin{array}{ll}
    (T^{[d]}_j{\bf x})(\epsilon)=T\mathbf{x}(\epsilon), & \hbox{$ \epsilon(j)=1$;} \\
    (T^{[d]}_j{\bf x})(\epsilon)=\mathbf{x}(\epsilon), & \hbox{$\epsilon(j)=0$.}
  \end{array}
\right.
\]

The {\em face group} of dimension $d$ is the group $\mathcal{F}^{[d]}(X)$ of
transformations of $X^{[d]}$ spanned by the face transformations.
The {\em parallelepiped group} of dimension $d$ is the group
$\mathcal{G}^{[d]}(X)$ spanned by the diagonal transformation and the face
transformations. We often write $\mathcal{F}^{[d]}$ and $\mathcal{G}^{[d]}$ instead of
$\mathcal{F}^{[d]}(X)$ and $\mathcal{G}^{[d]}(X)$ respectively. 
For convenience, we denote the orbit closure of $\mathbf{x}\in X^{[d]}$
under $\mathcal{F}^{[d]}$ by $\overline{\mathcal{F}^{[d]}}(\mathbf{x})$,
instead of $\overline{\mathcal{O}(\mathbf{x},\mathcal{F}^{[d]})}$.
Let $\Q^{[d]}_x(X)=\Q^{[d]}(X)\cap (\{x\}\times X^{2^d-1})$.

\begin{theorem}\cite{SY12}\label{property}
  Let $(X,T)$ be a minimal system and $d\in \N$. Then,
  \begin{enumerate}
  \item $(\mathbf{Q}^{[d]}(X),\mathcal{G}^{[d]})$ is a minimal system.
    \item $(\overline{\mathcal{F}^{[d]}}(x^{[d]}),\mathcal{F}^{[d]})$ is minimal for all $x\in X$.
    \item $\overline{\mathcal{F}^{[d]}}(x^{[d]})$ is the unique $\mathcal{F}^{[d]}$-minimal
    subset in $\mathbf{Q}^{[d]}_x(X)$ for all $x\in X$.
  \end{enumerate}
\end{theorem}

\subsection{Proximality and regionally proximality of higher order}
Let $(X,T)$ be a dynamical system. A pair $(x,y)\in X\times X$ is \emph{proximal} if
\[
\inf_{n\in \Z}\rho(T^nx,T^ny)=0
\]
and \emph{distal} if it is not proximal.
Denote by $\mathbf{P}(X)$ the set of all proximal pairs of $X$.
The dynamical system $(X,T)$ is \emph{distal} if $(x,y)$ is a distal pair whenever $x,y\in X$ are distinct.

An extension $\pi:(X,T)\to (Y,T)$ is \emph{proximal} if $R_{\pi}\subset \mathbf{P}(X)$.
\begin{definition}\label{def-rp}
Let $(X,T)$ be a dynamical system and $d\in \N$.
   The \emph{regionally proximal relation of order $d$} is the relation $\textbf{RP}^{[d]} (X)$
defined by: $(x,y)\in\textbf{RP}^{[d]}(X)$ if
and only if for every $\delta>0$, there
exist $x',y'\in X$ and $\vec{n}\in \Z^d$ such that:
$\rho(x,x')<\delta,\rho(y,y')<\delta$, and
\[
\rho(  T^{\vec{n}\cdot\epsilon} x', T^{\vec{n}\cdot\epsilon}  y')<\delta\;
\text{for all}\; \epsilon\in \{0,1\}^d\backslash\{ \vec{0}\}.
\]

We say $(X,T)$ is a \emph{system of order $d$}
if $\RP^{[d]}(X)$ is trivial.
\end{definition}

It is easy to see that $\mathbf{RP}^{[d]}(X)$ is a closed and invariant relation.
Note that
\[\mathbf{P}(X)\subset
\ldots\subset \mathbf{RP}^{[d+1]}(X)\subset \mathbf{RP}^{[d]}(X)\subset
\ldots \subset\mathbf{RP}^{[2]}(X)\subset \mathbf{RP}^{[1]}(X).
\]

\begin{theorem}\cite{SY12}\label{cube-minimal}
  Let $(X,T)$ be a minimal system and $d\in \N$.
  Then,
  \begin{enumerate}
    \item $(x,y)\in \mathbf{RP}^{[d]}(X)$ if and only if $(x,y,y,\ldots,y)=(x,y^{[d+1]}_*)\in \mathbf{Q}^{[d+1]}(X)$
    if and only if $(x,y,y,\ldots,y)=(x,y^{[d+1]}_*)\in \overline{\mathcal{F}^{[d+1]}}(x^{[d+1]})$.
    \item $\mathbf{RP}^{[d]}(X)$ is an equivalence relation.
  \end{enumerate}
\end{theorem}

The regionally proximal relation of order $d$ allows us to construct the maximal
factor of order $d$
of a minimal system. That is, any factor of order $d$
factorizes through this system.

\begin{theorem}\label{lift-property}\cite{SY12}
  Let $\pi :(X,T)\to (Y,T)$ be the factor map between minimal systems and $d\in \N$. Then,
  \begin{enumerate}
    \item $(\pi \times \pi) \mathbf{RP}^{[d]}(X)=\mathbf{RP}^{[d]}(Y)$.
    \item $(Y,T)$ is a system of order $d$ if and only if $\mathbf{RP}^{[d]}(X)\subset R_\pi$.
  \end{enumerate}

In particular, the quotient of $(X,T)$ under $\mathbf{RP}^{[d]}(X)$
is the maximal factor of order $d$ of $X$.
\end{theorem}

It follows that for any minimal system $(X,T)$,
\[
\mathbf{RP}^{[\infty]}(X)=\bigcap_{d\geq1}\mathbf{RP}^{[d]}(X)
\]
is a closed invariant equivalence relation.

Now we formulate the definition of systems of order $\infty$.

\begin{definition}
    A minimal system $(X,T)$ is a \emph{system of order $\infty$},
  if the equivalence relation $\mathbf{RP}^{[\infty]}(X)$ is trivial,
  i.e. coincides with the diagonal.
\end{definition}

Let $(X,T)$ be a dynamical system, set
$\RP^{[d]}[x]=\{y\in X:(x,y)\in \RP^{[d]}(X)\},$
where $d\in \N\cup \{\infty\}$.

\begin{definition}
 Let $(X,T)$ be a minimal system and $d\in \N\cup \{\infty\}$.
 A point $x\in  X$ is called
a \emph{$d$-step almost automorphic point} if $\RP^{[d]}[x] = \{x\}.$

A minimal system $(X,T)$ is called \emph{$d$-step almost automorphic} if
it has a $d$-step almost automorphic point.
 \end{definition}

Almost automorphic systems of higher order were studied systematically in \cite{HSY16}, in particular
we have

\begin{prop}\cite[Theorem 8.13]{HSY16}
 Let $(X,T)$ be a minimal system. Then $(X,T)$ is a
$d$-step almost automorphic system for some $d\in \N\cup \{\infty\}$  if and only if it is an almost
one-to-one extension of its maximal factor of order $d$.
\end{prop}

\subsection{Independence}
The notion of \emph{independence} was firstly introduced and studied in \cite{KL07}.
It corresponds to a modification of the notion of \emph{interpolating set}
studied in \cite{GW95,HY06}.

\begin{definition}
Let $(X,T)$ be a dynamical system. Given a tuple $\mathcal{A} = (U_0,U_1,\ldots,U_k )$ of subsets
of $X$ we say that a subset $F\subset \Z$  is an \emph{independence set} for $\mathcal{A}$ if for any nonempty
finite subset $ J\subset  F$ and any $s=(s(j):j\in J)\in \{0,1,\ldots,k\}^J$  we have
\[
\bigcap_{j\in J}T^{-j}U_{s(j)}\neq \emptyset.
\]
We shall denote the collection of all independence sets for $\mathcal{A}$ by $\I(U_0,U_1 ,\ldots,U_k )$.
\end{definition}


Now we define $\text{IN}^{[d]}$-pairs.

\begin{definition}\label{def-IN-d}
Let $(X,T)$ be a dynamical system and $d\in \N$.
A pair $(x_0 ,x_1 ) \in X \times X$ is called an \emph{$\text{IN}^{[d]}$-pair}
if for any $k\in \N$ and any neighborhoods $U_0,U_1$ of $x_0$ and $x_1$ respectively,
 there exist integers $p_j^{(i)},1\leq i\leq k,1\leq j\leq d$ such that
\[
\bigcup_{i=1}^k\{ p_1^{(i)}\epsilon(1)+\ldots+p_d^{(i)}\epsilon(d):\epsilon\in \{0,1\}^d\}\backslash \{0\}\subset \I(U_0,U_1).
\]
Denote by $\mathrm{IN}^{[d]} (X)$ the set of all $\mathrm{IN}^{[d]}$-pairs of $(X,T)$.
\end{definition}

\begin{remark}\label{include}
It is easy to see that for a dynamical system,
any $\mathrm{IN}^{[d]}$-pair
is regionally proximal of order $d$,
  sequence entropy pairs coincide with
  $\mathrm{IN}^{[1]}$-pairs
   and any $\I_{fip}$-pair is an $\mathrm{IN}^{[d]}$-pair for every $d\in \N$.
\end{remark}

\subsection{A criterion to be an $\mathrm{IN}^{[d]}$-pair}\label{criterion}
We characterize $\mathrm{IN}^{[d]}$-pairs using
dynamical cubespaces.

Let $d,k\in \N$.
We fix an enumeration,
$\omega_1,\ldots,\omega_{2^d-1}$ of all elements of $\{0,1\}^d\backslash\{\vec{0}\}$.
For $1\leq i\leq k,1\leq j\leq 2^d-1$, let
\[
F_{ij}= \left\{ \epsilon\in \{0,1\}^{k(2^d+d)}:
\begin{gathered}
\epsilon({k(i-1)+j})=1, \text{ and }
\\ \epsilon({k2^d+d(i-1)+s})=\omega_j(s),\;1\leq s\leq d
 \end{gathered}
\right\}.
\]
For $t_j\in \{0,1\}^{2^d-1},1\leq j\leq k$, let $\hat{\theta}=\hat{\theta}(t_1,\ldots,t_k)\in \{0,1\}^{k(2^d+d)}$ such that
\[
\hat{\theta}(n)=
\begin{cases}
 t_i(j), &n=k(i-1)+j,\;1\leq i\leq k,1\leq j\leq 2^d-1; \\
  0,& \mathrm{otherwise},
\end{cases}
\]
for $1\leq a\leq k,1\leq b\leq 2^d-1$, let
$\theta=\theta(t_1,\ldots,t_k,a,b)\in \{0,1\}^{k(2^d+d)}$ such that
\[
\theta(n)=
\begin{cases}
 t_i(j), &n=k(i-1)+j,\;1\leq i\leq k,1\leq j\leq 2^d-1; \\
  \omega_b(s), & n=k2^d+d(a-1)+s, \;1\leq s\leq d;\\
  0,& \mathrm{otherwise}.
\end{cases}
\]

Let
\[
\Theta_{k,d}=\{
\theta=\theta(t_1,\ldots,t_k,a,b):
1\leq a\leq k,1\leq b\leq 2^d-1,
t_j\in \{0,1\}^{2^d-1},1\leq j\leq k\}.
\]

It is easy to check that
$\theta=\theta(t_1,\ldots,t_k,a,b)\in F_{ij}$ if and only if
$a=i,b=j$ and $t_a(b)=1$.

\begin{lemma}\label{INN}
Let $(X,T)$ be a minimal system and
 $d\in \N,x_0,x_1\in X$ with $x_0\neq x_1$.
For any $k\in \N$,
if there is some $\mathbf{x}\in\Q^{[k(2^d+d)]}(X)$ such that
 $\mathbf{x}(\theta)=x_{t_a(b)}$
for any $\theta\in \Theta_{k,d}$,
then $(x_0,x_1)$ is an $\mathrm{IN}^{[d]}$-pair.
\end{lemma}

\begin{proof}
For $i=0,1$,
let $U_i$ be a neighborhood of $x_i$ and
choose $\delta>0$ with $B(x_i,\delta)=\{y\in X\colon \rho(x_i,y)<\delta\}\subset U_i$.

Let $k\in \N$ and let $\mathbf{x}\in\Q^{[k(2^d+d)]}(X)$ such that
 $\mathbf{x}(\theta)=x_{t_a(b)}$
 for any $\theta\in \Theta_{k,d}$.

By Theorem \ref{property} there exist
$$\vec{n}=(n_1,\ldots,n_{k2^d},m_1^{(1)},\ldots,m_d^{(1)},\ldots,m_1^{(k)},\ldots,m_d^{(k)})\in \Z^{k(2^d+d)},
$$
$n\in \Z $ and $x\in X$ such that
 \begin{equation}\label{belong1}
 \rho(T^{n+\vec{n}\cdot \epsilon}x,\mathbf{x}(\epsilon))<\delta\; \text{for all} \;\epsilon\in \{0,1\}^{k(2^d+d)}.
  \end{equation}

For $1\leq i\leq k$, set $\vec{m}_i=(m_1^{(i)},\ldots,m_d^{(i)})$.
Recall that $\mathbf{x}(\theta)=x_{t_a(b)}$ and
$\vec{n}\cdot\theta= \vec{n}\cdot\hat{\theta}+\vec{m}_a\cdot \omega_{b}$,
thus by (\ref{belong1}) we get that
\[
T^{n+ \vec{n}\cdot\hat{\theta}}x\in
 T^{-\vec{m}_a\cdot \omega_{b}} U_{t_a(b)}.
\]
Moreover,
\[
T^{ n+\vec{n}\cdot\hat{\theta}}x\in
\bigcap_{i=1}^k \bigcap_{j=1}^{2^d-1} T^{-\vec{m}_i\cdot \omega_{j}} U_{t_i(j)},
\]
which implies that
\[
\bigcup_{i=1}^k\{ \vec{m}_i\cdot \epsilon:\epsilon\in \{0,1\}^d\}\backslash \{0\}\subset \I(U_0,U_1).
\]

As $k$ is arbitrary, we
conclude that $(x_0,x_1)$ is an $\mathrm{IN}^{[d]}$-pair.
\end{proof}

\subsection{Nilpotent groups, nilmanifolds and nilsystems}
Let $L$ be a group.
For $g,h\in L$, we write $[g,h]=ghg^{-1}h^{-1}$ for the commutator of $g$ and $h$,
we write $[A,B]$ for the subgroup spanned by $\{[a,b]:a\in A,b\in B\}$.
The commutator subgroups $L_j,j\geq1$, are defined inductively by setting $L_1=L$
and $L_{j+1}=[L_j,L]$.
Let $k\geq 1$ be an integer.
We say that $L$ is \emph{k-step nilpotent} if $L_{k+1}$ is the trivial subgroup.

\medskip

Let $L$ be a $k$-step nilpotent Lie group and $\Gamma$ a discrete cocompact subgroup of $L$.
The compact manifold $X=L/\Gamma$ is called a \emph{k-step nilmanifold.}
The group $L$ acts on $X$ by left translations and we write this action as $(g,x)\mapsto gx$.
Let $\tau\in L$ and $T$ be the transformation $x\mapsto \tau x$ of $X$.
Then $(X,T)$ is called a \emph{k-step nilsystem}.

We also make use of inverse limits of nilsystems and so we recall the definition of an inverse limit
of systems (restricting ourselves to the case of sequential inverse limits).
If $\{(X_i,T_i)\}_{i\in \N}$ are systems with $\text{diam}(X_i)\leq 1$ and $\phi_i:X_{i+1}\to X_i$
are factor maps, the \emph{inverse limit} of the systems is defined to be the compact subset of $\prod_{i\in \N}X_i$
given by $\{(x_i)_{i\in \N}:\phi_i(x_{i+1})=x_i,i\in \N\}$,
which is denoted by $\lim\limits_{\longleftarrow}\{ X_i\}_{i\in \N}$.
It is a compact metric space endowed with the distance $\rho(x,y)=\sum_{i\in \N}1/ 2^i \rho_i(x_i,y_i)$.
We note that the maps $\{T_i\}$ induce a transformation $T$ on the inverse limit.

The following structure theorem characterizes inverse limits of nilsystems using dynamical cubespaces.

\begin{theorem}[Host-Kra-Maass]\cite[Theorem 1.2]{HKM}\label{description}
  Assume that $(X,T)$ is a minimal system
  and let $d\geq2$ be an integer. The following properties are equivalent:
  \begin{enumerate}
    \item If $\mathbf{x},\mathbf{y}\in \mathbf{Q}^{[d]}(X)$ have $2^d-1$ coordinates in common, then $\mathbf{x}=\mathbf{y}$.
    \item If $x,y\in X$ are such that $(x,y,\ldots,y)\in  \mathbf{Q}^{[d]}(X)$, then $x=y$.
    \item The system $(X,T)$ is an inverse limit of $(d-1)$-step minimal nilsystems.
  \end{enumerate}
\end{theorem}
This result shows that a minimal system is a system of order $d$
if and only if it is an inverse limit of minimal $d$-step nilsystems.

\begin{theorem}\cite[Theorem 3.6]{DDMSY13}\label{system-of-order}
  A minimal system $(X,T)$
  is a system of order $\infty$ if and only if it is an inverse limit of minimal nilsystems.
\end{theorem}

\section{$k$-regionally proximal relation of higher order}
In this section, we discuss the $k$-regionally proximal relation of high order.

\begin{definition}\label{def-k-rp}
  Let $(X,T)$ be a dynamical system and $d\in \N$.
  For $k\geq  2$, a $k$-tuple $(x_1,\ldots,x_k)\in X^k$
  is said to be \emph{$k$-regionally proximal of order $d$} if for any $\delta>0$,
  there exist $x_i'\in X,1\leq i\leq k$ and $\vec{n}\in \Z^d$
  such that $\rho(x_i,x_i')<\delta,1\leq i\leq k$ and
  \[
 \max_{1\leq i<j\leq k} \rho(T^{\vec{n}\cdot \epsilon}x_i',T^{\vec{n}\cdot \epsilon}x_j')<\delta\;
\text{ for all}\; \epsilon\in \{0,1\}^d\backslash \{\vec{0}\}.
  \]
  \end{definition}

In the proof of the following theorem,
we will use enveloping semigroups in abstract topological dynamical systems.
For more details, see Appendix \ref{sectionA}.

\begin{theorem}\label{main-thm1}
Let $(X,T)$ be a minimal system and let $d,k\in \N$ with $k\geq 2$.
For points $x,x_i\in X,1\leq i\leq k$,
if $(x,x_{i})$ is regionally proximal of order $d$ for all $i$,
then $(x_1,\ldots,x_k)$ is $k$-regionally proximal of order $d$.
\end{theorem}

\begin{proof}
Let $d,k\in \N$ with $k\geq 2$.
Fix $x\in X$ and
let $x_i\in \mathbf{RP}^{[d]}[x],1\leq i\leq k$.

We will show that
$(x_1,\ldots,x_k)$ is $k$-regionally proximal of order $d$. 

\medskip

\noindent {\bf Claim 1:}
Let $y\in \RP^{[d]}[x]$ and let
$\mathbf{y}\in X^{[d+1]}$ such that
\[
\mathbf{y}(\epsilon)=
\begin{cases}
y, &  \epsilon=(0,\ldots,0,1);\\
x, &   \hbox{otherwise,}
\end{cases}
\]
then $\mathbf{y} \in\overline{\mathcal{F}^{[d+1]}}(x^{[d+1]})$.

\begin{proof}[Proof of Claim 1]
As $(x,y)\in \RP^{[d]}(X)\subset  \RP^{[d-1]}(X)$,
we have  $(x,y^{[d]}_*)\in \overline{\mathcal{F}^{[d]}}(x^{[d]})$
by Theorem \ref{cube-minimal}.
Notice that $(\overline{\mathcal{F}^{[d]}}(x^{[d]}),\mathcal{F}^{[d]})$
is minimal by Theorem \ref{property},
then
there is some sequence $\{\vec{n}_j\}_{j\in \N}\subset \Z^d$
such that
\begin{equation}\label{limit}
  (T^{\vec{n}_j\cdot \epsilon}:\epsilon\in \{0,1\}^d)(x,y^{[d]}_*)\to x^{[d]},
\end{equation}
as $j\to \infty$.
Let $\sigma$ be the map from $\Z^{d}$ to $\Z^{d+1}$ such that
\[
\vec{n}=(n_1,\ldots,n_{d})\mapsto \sigma(\vec{n})=(n_1,\ldots,n_{d},0).
\]
Again by Theorem \ref{cube-minimal},
$(x,y^{[d+1]}_*)\in \overline{\mathcal{F}^{[d+1]}}(x^{[d+1]})$.
Then by (\ref{limit}) we have
\[
(T^{\sigma( \vec{n}_{j})\cdot \omega}:\omega\in \{0,1\}^{d+1})(x,y^{[d+1]}_*)
\to \mathbf{y},
\]
as $j\to \infty$,
which implies that $\mathbf{y}\in \overline{\mathcal{F}^{[d+1]}}(x^{[d+1]})$.
\end{proof}

For $1\leq i\leq k$ and $s=0,1$, let
\[
F_i^s=\{\epsilon\in\{0,1\}^{d+k}:\epsilon(j)=0,1\leq j\leq d,\;\epsilon(d+i)=s\}.
\]

\noindent {\bf Claim 2:}
For every $1\leq i\leq k$, there is
$\mathbf{p}_i\in E(\overline{\mathcal{F}^{[d+k]}}(x^{[d+k]}),\mathcal{F}^{[d+k]})$ such that
\begin{enumerate}
  \item $\mathbf{p}_i(\epsilon)=\mathrm{id},\epsilon\in F_i^0$;
  \item $\mathbf{p}_i(\epsilon)x=x_i,\epsilon\in F_i^1$;
  \item $\mathbf{p}_i(\epsilon)x=x,\epsilon\in \{0,1\}^{d+k}\backslash F_i^1$.
\end{enumerate}

\begin{proof}[Proof of Claim 2]
Let $i\in \{1,\ldots,k\}$ and let
$\mathbf{a}_i\in X^{[d+1]}$ such that
\[
\mathbf{a}_i(\epsilon)=
\begin{cases}
x_i, &  \epsilon=(0,\ldots,0,1);\\
x, &   \hbox{otherwise,}
\end{cases}
\]
then
 $\mathbf{a}_i\in \overline{\mathcal{F}^{[d+1]}}(x^{[d+1]})$
 by Claim 1.
 Notice that $(\overline{\mathcal{F}^{[d+1]}}(x^{[d+1]}),\mathcal{F}^{[d+1]})$
is minimal,
then there is some sequence $\{\vec{n}^{(l)}=(n_1^{(l)},\ldots,n_{d+1}^{(l)})\}_{l\in \N}\subset \Z^{d+1}$ such that
\begin{equation}\label{convengence}
  (T^{\vec{n}^{(l)}\cdot \epsilon}x:\epsilon\in \{0,1\}^{d+1})\to \mathbf{a}_i,
\end{equation}
as $ l\to\infty$.
For $l\in \N$, let $\vec{m}^{(l)}=(m_1^{(l)},\ldots,m_{d+k}^{(l)})\in \Z^{d+k}$ such that
\[
m_j^{(l)}=
\left\{
  \begin{array}{ll}
   n_j^{(l)}, & \hbox{$j=1,\ldots,d$;} \\
    n_{d+1}^{(l)}, & \hbox{$j=d+i$;}\\
    0,& \hbox{otherwise.}
  \end{array}
\right.
\]
Then by (\ref{convengence}) we have that
\begin{enumerate}
  \item for $\epsilon\in F_i^0,T^{\vec{m}^{(l)}\cdot \epsilon}=T^0=\mathrm{id}$;
  \item for $\epsilon\in F_i^1,T^{\vec{m}^{(l)}\cdot \epsilon}x=T^{n_{d+1}^{(l)}}x\to x_i$, as $l\to \infty$;
  \item for $\epsilon\in \{0,1\}^{d+k}\backslash F_i^1,T^{\vec{m}^{(l)}\cdot \epsilon}x=
  T^{\vec{n}^{(l)}\cdot \tilde{\epsilon}}x\to x$, as $l\to \infty$, where $\tilde{\epsilon}\in \{0,1\}^{d+1}$
  with $\tilde{\epsilon}(i)=\epsilon(i),1\leq i\leq d$ and $\tilde{\epsilon}(d+1)=\epsilon(d+i)$.
\end{enumerate}

Now assume that
\[
(T^{\vec{m}^{(l)}\cdot \epsilon}:\epsilon\in \{0,1\}^{d+k})\to \mathbf{p}_i
\]
in $E(\overline{\mathcal{F}^{[d+k]}}(x^{[d+k]}),\mathcal{F}^{[d+k]})$ pointwise.

It is easy to check that $\mathbf{p}_i$ meets the requirement.
\end{proof}

\medskip

Now let $\mathbf{y}=\mathbf{p}_k\cdots\mathbf{p}_1 x^{[d+k]}$.
For $1\leq i\leq k$, let $\omega_i=\min F_i^1$ and let
\[
F=\{\epsilon\in \{0,1\}^{d+k}:\sum_{j=1}^{d}\epsilon(j)>0\}.
\]

\noindent {\bf Claim 3:}
 $\mathbf{y}\in \overline{\mathcal{F}^{[d+k]}}(x^{[d+k]})$ and
\begin{enumerate}
  \item $\mathbf{y}(\omega_i)=x_i,\; 1\leq i\leq k$;
  \item $\mathbf{y}(\epsilon)=x,\epsilon\in F$.
\end{enumerate}

\begin{proof}[Proof of Claim 3]
Notice that $\omega_i\in F_j^0$ for any $i\neq j$,
thus $\mathbf{p}_j(\omega_i)=\mathrm{id}$ by property (1) of Claim 2.
By property (2) of Claim 2, we have $\mathbf{p}_i(\omega_i)x=x_i$. This shows that
\[
\mathbf{y}(\omega_i)=
\mathbf{p}_k(\omega_i)\cdots \mathbf{p}_1(\omega_i)x=\mathbf{p}_i(\omega_i)x=x_i.
\]

Let $\epsilon\in F$, then $\epsilon\notin \cup_{i=1}^k F_i^1$.
By property (3) of Claim 2, $\mathbf{p}_i(\epsilon)x=x$ for every $i$ and thus
we get that
\[
\mathbf{y}(\epsilon)=\mathbf{p}_k(\epsilon)\cdots\mathbf{p}_1(\epsilon)x=x.
\]

This shows Claim 3.
\end{proof}
Fix $\delta>0$.
As $\mathbf{y}\in \overline{\mathcal{F}^{[d+k]}}(x^{[d+k]})$,
there is some $\vec{m}=(m_1,\ldots,m_{d+k})\in \Z^{d+k}$ such that
for all $\epsilon\in \{0,1\}^{d+k}$,
\begin{equation}\label{inequation}
\rho(T^{\vec{m}\cdot \epsilon}x,\mathbf{y}(\epsilon))<\delta.
\end{equation}

Let $x_i'=T^{\vec{m}\cdot \omega_i}x,1\leq i\leq k$ and $\vec{n}=(m_1,\ldots,m_d)$.
By (\ref{inequation}) and property (1) of Claim 3, for $1\leq i\leq k$, we have
\[
\rho(x_i',x_i)=\rho(T^{\vec{m}\cdot \omega_i}x,\mathbf{y}(\omega_i))<\delta.
\]

For $\epsilon\in\{0,1\}^d\backslash \{\vec{0}\}$,
put $\hat{\epsilon}\in \{0,1\}^{d+k}$ such that
\[
\hat{\epsilon}(j)=
\begin{cases}
\epsilon(j), &  1\leq j\leq d;\\
0, & d+1\leq j\leq d+k,
\end{cases}
\]
then $\hat{\epsilon}+\omega_i\in F$ for $1\leq i\leq k$.
Moreover we have that
\[
\rho(T^{\vec{n}\cdot \epsilon}x_i',x)=
\rho(T^{\vec{n}\cdot \epsilon+\vec{m}\cdot \omega_i}x,x)=\rho(T^{\vec{m}\cdot (\hat{\epsilon}+\omega_i)}x,
\mathbf{y}(\hat{\epsilon}+\omega_i))<\delta ,
\]
by (\ref{inequation}) and property (2) of Claim 3
which implies that $(x_1,\ldots,x_k)$
is $k$-regionally proximal of order $d$.

The proof is completed.
\end{proof}

\section{Independence and minimal nilsystems}
The main goal of this section is to study the $\mathrm{IN}^{[d]}$-pairs in minimal nilsystems.
It turns out that for a minimal nilsystem, any
regionally proximal of order $d$ pair is an $\mathrm{IN}^{[d]}$-pair.
We start by recalling some basic results in nilsystems.
For more details and proofs, see \cite{AGF63,PW}.

\medskip

If $G$ is a nilpotent Lie group,
let $G^0$ denote the connected component of its
unit element $1_G$.
In the sequel, $s\geq2$ is an integer and
$(X=G/\Gamma,T)$ is a minimal $s$-step nilsystem.
We let $\tau$ denote the element of $G$ defining the transformation $T$.
If $(X,T)$ is minimal, let $G'$ be the subgroup of $G$
spanned by $G^0$ and $\tau$ and let $\Gamma'=\Gamma\cap G'$,
then we have that $G=G'\Gamma$.
Thus the system $(X,T)$ is conjugate to the system
$(X',T')$
where $X'=G'/\Gamma'$ and $T'$ is the translation by $\tau$ on $X'$.
Therefore,
without loss of generality, we can restrict to the case $G$ is spanned by $G^0$
and $\tau$.

We fix an enumeration,
$F_1,F_2,\ldots,F_{2^d}$ of all upper faces of $\{0,1\}^d$,
ordered such that codim$(F_i)$ is nondecreasing with $i$.
Then $F_1=\{0,1\}^d$, the upper faces of codimension 1 are $F_2,\ldots,F_{d+1}$.

If $F$ is a face of $\{0,1\}^d$, for $g\in G$ we define
$g^{(F)}\in G^{[d]}$ by
\[
(g^{(F)})(\epsilon)=
\begin{cases}
  g, &  \epsilon\in F; \\
 1_G, &  \epsilon\not\in F.
\end{cases}
\]

Denote by $\mathcal{HK}^{[d]}$ the subgroup of $G^{[d]}$ spanned by
\[
\{g^{(F_i)}:g\in G,\;  1\leq i\leq d+1 \}.
\]

\begin{lemma}\label{rational}\cite[Chapter 12]{HK18}
   $\mathcal{HK}^{[d]}$ is a rational subgroup of $G^{[d]}$.
\end{lemma}

Lemma \ref{rational} means that $\Gamma^{[d]}\cap \mathcal{HK}^{[d]}$ is cocompact in $\mathcal{HK}^{[d]}$,
allowing us to define an $s$-step nilmanifold
\[
\widetilde{X}^{[d]}=
\frac{\mathcal{HK}^{[d]}}{\Gamma^{[d]}\cap \mathcal{HK}^{[d]}}.
\]

\begin{lemma}\cite[Chapter 12]{HK18}\label{nilmanifold}
The nilmanifold
$\widetilde{X}^{[d]}$
is equal to $\mathbf{Q}^{[d}(X)$.
\end{lemma}

By Lemma \ref{nilmanifold}
we can view $\mathbf{Q}^{[d]}(X)$ as a nilsystem
and it is also $\mathcal{HK}^{[d]}$-invariant.

\begin{lemma}\label{commutator}\cite[Chapter 12]{HK18}
Let $F$ be a face of $\{0,1\}^d$ and let $g\in G_{\mathrm{codim}(F)}$,
then $g^{(F)}\in \mathcal{HK}^{[d]}$.
\end{lemma}

The following corollary follows from Lemmas \ref{nilmanifold} and \ref{commutator} immediately.
\begin{cor}\label{invariant}
Let $F$ be a face of $\{0,1\}^d$ and let $g\in G_{\mathrm{codim}(F)}$,
then $g^{(F)}\mathbf{x}\in \Q^{[d]}(X)$
for every $\mathbf{x}\in \Q^{[d]}(X)$.
\end{cor}

Now we are in position to show the main result of this section.
When proving,
we omit the nilpotency class as it is not important.

\begin{theorem}\label{ind-nilsystem}
Let $(X=G/\Gamma,T)$ be a minimal nilsystem.
For $x\in X$ and $g\in G_{d+1}$,
if $x\neq gx$,
then
$(x,gx)$ is an $\mathrm{IN}^{[d]}$-pair.
\end{theorem}
\begin{proof}
Let $x\in X$ and
$g\in G_{d+1}$ with $x\neq gx$.
Put $x_0=x$ and $x_1=gx$.
For $i=0,1$,
let $U_i$ be a neighborhood of $x_i$ and
choose $\delta>0$ with $B(x_i,\delta)\subset U_i$.

We fix an enumeration,
$\omega_1,\ldots,\omega_{2^d-1}$ of all elements of $\{0,1\}^d\backslash\{\vec{0}\}$.

Let $k\in \N$. For $1\leq i\leq k,1\leq j\leq 2^d-1$, let
\[
F_{ij}= \left\{ \epsilon\in \{0,1\}^{k(2^d+d)}:
\begin{gathered}
\epsilon({k(i-1)+j})=1, \text{ and }
\\ \epsilon({k2^d+d(i-1)+s})=\omega_j(s),\;1\leq s\leq d
 \end{gathered}
\right\}.
\]


Notice that $F_{ij}$ is a face of $\{0,1\}^{k(2^d+d)}$ of codimension $d+1$ for any $i,j$.
It follows from Lemma \ref{commutator} that $g^{(F_{ij})} \in\mathcal{HK}^{[k(2^d+d)]}$.
Thus by Corollary \ref{invariant} we have that
\begin{equation}\label{eqd}
 \mathbf{x}= \left(\prod_{i=1}^{k}\prod_{j=1}^{2^d-1}g^{(F_{ij})}\right)x^{[k(2^d+d)]}\in \Q^{[k(2^d+d)]}(X).
\end{equation}

Recall that
for any $\theta=\theta(t_1,\ldots,t_k,a,b)\in \Theta_{k,d}$,
we have
$\theta\in F_{ij}$ if and only if
$a=i,b=j$ and $t_a(b)=1$ (see subsection \ref{criterion}),
thus
by (\ref{eqd}) we get that
\[
\mathbf{x}(\theta)=
\left(\prod_{\substack{ 1\leq i\leq k,1\leq j\leq 2^d -1 \\  \theta\in F_{ij} \;}}g\right)x=x_{t_a(b)}.
\]

By Lemma \ref{INN}, we deduce that
$(x,gx)=(x_0,x_1)$
is an $\mathrm{IN}^{[d]}$-pair.
\end{proof}


We refer to \cite{DDMSY13} for the following description of maximal factors of higher order
 of minimal nilsystems.

\begin{lemma}\label{nilfact}
 For $1\leq r\leq s$, if $X_r$ is the maximal factor of order $r$ of $X$,
  then $X_r$ has the form $G/(G_{r+1}\Gamma)$, endowed with the translation by the projection
of $\tau $ on $G/G_{r+1}$.
\end{lemma}

The following corollary follows from Theorem \ref{ind-nilsystem} and Lemma \ref{nilfact} immediately.

\begin{cor}\label{nil-rp}
  Let $(X,T)$ be a minimal nilsystem and $d\in \N$.
    Then $(x_0,x_1)\in\mathrm{IN}^{[d]}(X)$
if and only if $(x_0,x_1)\in \RP^{[d]}(X)$.
\end{cor}

\begin{cor}\label{inverse-limit}
 Let $(X,T)$ be an inverse limit of minimal nilsystem and $d\in \N$.
     Then $(x_0,x_1)\in\mathrm{IN}^{[d]}(X)$
if and only if $(x_0,x_1)\in \RP^{[d]}(X)$.
\end{cor}

\begin{proof}
Let $(x_0,x_1)\in \mathbf{RP}^{[d]}(X)\backslash\Delta_X $.
By the definition of $\mathrm{IN}^{[d]}$-pairs,
it suffices to show that $(x_0,x_1)\in \mathrm{IN}^{[d]}(X)$.

Assume that $X_i$ is a minimal nilsystem for every $i\in \N$,
set $X=\lim\limits_{\longleftarrow}\{X_i\}_{i\in \N}$
and assume that $\pi_i :X\to X_i$ and $\pi_{i,j}:X_j\to X_i$ are the factor maps.

Set $x_s^i=\pi_i(x_s),i\in \N,s=0,1$,
then $\pi_{i,j}(x_s^j)=x_s^i$
and
there is some $n\in \N$ such that $x_0^n\neq x_1^n$.
For $j\geq n$ and $s=0,1$, we have
$x_s^n=\pi_{n,j}(x_s^j)$ which implies $x_0^j\neq x_1^j$.
Without loss of generality, we may assume $x_0^i\neq x_1^i$ for all $i\in \N$.

Let $k\in \N$.
It follows from Theorem \ref{lift-property} that $(x_0^i,x_1^i)\in \RP^{[d]}(X_i) $
for all $i\in \N$.
By Theorem \ref{ind-nilsystem} and Lemma \ref{nilfact},
for every $i\in \N$ there exists some $\mathbf{x}_i\in\Q^{[k(2^d+d)]}(X_i)$ such that
 $\mathbf{x}_i(\theta)=x_{t_a(b)}^i$
for all $\theta \in \Theta_{k,d}$. 
Notice that $\Q^{[k(2^d+d)]}(X)$ is an inverse limit of the sequence $\{\Q^{[k(2^d+d)]}(X_i)\}_{i\in \N}$,
thus for every $i\in \N$ there exists some $\widetilde{\mathbf{x}}_i\in\Q^{[k(2^d+d)]}(X)$ such that
\[
\pi_i^{[k(2^d+d)]}(\widetilde{\mathbf{x}}_i)=\mathbf{x}_i.
\]
Without loss of generality, assume that $\widetilde{\mathbf{x}}_i \to \mathbf{x}$ as $i\to \infty$
for some $\mathbf{x}\in \Q^{[k(2^d+d)]}(X)$.

We claim that $\mathbf{x}(\theta)=x_{t_a(b)}$ for all $\theta\in \Theta_{k,d}$.

Actually, for any $\theta \in \Theta_{k,d}$ and $i\leq j$ we have
\[
\pi_i(\widetilde{\mathbf{x}}_j(\theta))=\pi_{i,j}\circ \pi_j(\widetilde{\mathbf{x}}_j(\theta))
=\pi_{i,j}( \mathbf{x}_j(\theta))=\pi_{i,j}( x_{t_a(b)}^j)=x_{t_a(b)}^i.
\]
By letting $j$ go to infinity and the continuity of $\pi_i$,
we get $\pi_i(\mathbf{x}(\theta))=x_{t_a(b)}^i$ for all $i\in\N$.
This shows the claim
and thus
$(x_0,x_1)\in \mathrm{IN}^{[d]}(X)$ by Lemma \ref{INN}.

This completes the proof.
\end{proof}

\section{The structure of minimal systems without nontrivial $\mathrm{IN}^{[d]}$-pairs}

In this section we discuss the structure of minimal systems without nontrivial
$\mathrm{IN}^{[d]}$-pairs. We will show that such systems are almost one-to-one extensions of
their maximal factors of order $d$.

We start from the following useful lemma which can be found in the proof of \cite[Theorem 3.1]{SY12}.

\begin{lemma}\label{minimal}
  Let $(X,T)$ be a dynamical system and $d\in \N$.
  If $\mathbf{x}\in X^{[d]}$ is an $\mathrm{id }\times T^{[d]}_*$-minimal point,
  then it is an $\mathcal{F}^{[d]}$-minimal point.
\end{lemma}

\begin{lemma}\label{replace}
Let $(X,T)$ be a minimal system and $d\in \N$.
For $\omega\in\{0,1\}^d$ and $\mathbf{x}\in \Q^{[d]}(X)$,
let $\mathbf{y}\in X^{[d]}$ such that $\mathbf{y}(\epsilon)=\mathbf{x}(\epsilon)$
if $\epsilon\in \{0,1\}^d\backslash \{\omega\}$
and $(\mathbf{x}(\omega),\mathbf{y}(\omega))\in \RP^{[\infty]}(X)$.
If $\mathbf{y}$ is a $T^{[d]}$-minimal point,
then $\mathbf{y}\in \Q^{[d]}(X)$.
\end{lemma}

\begin{proof}
Let $\mathbf{x}(\omega)=x,\mathbf{y}(\omega)=y$ and $(x,y)\in \RP^{[\infty]}(X)$.

\medskip

\noindent {\bf Case 1: $\omega=\vec{0}$.}

As $\mathbf{x}\in \Q^{[d]}(X)$,
there exists some sequence $\{S_i\}_{i\in \N}\subset \mathcal{F}^{[d]}$ such that
$S_i \mathbf{x}\to x^{[d]}$ as $i\to \infty$ by Theorem \ref{property}.
At the same time, we have that $S_i\mathbf{y}\to (y,x^{[d]}_*)\in \Q^{[d]}(X)$ as $i\to \infty$.
By our hypothesis, $\mathbf{y}$ is a $T^{[d]}$-minimal point, so $\mathbf{y}$ is also an $\mathrm{id }\times T^{[d]}_*$-minimal point.
 Thus by Lemma \ref{minimal},
 $\mathbf{y}$ is an $\mathcal{F}^{[d]}$-minimal point which implies that
$\mathbf{y}$ also belongs to the orbit closure of $(y,x^{[d]}_*)$ under the $\mathcal{F}^{[d]}$-action,
and thus $\mathbf{y}\in \Q^{[d]}(X)$.

\medskip

\noindent {\bf Case 2: $\omega\neq\vec{0}$.}

Recall that
$\mathbf{Q}^{[d]}(X)$ is invariant under the Euclidean permutations of $X^{[d]}$.
We can choose some Euclidean permutation $f$ such that $f(\mathbf{y})(\vec{0})=\mathbf{y}(\omega)$.

Now we have $f(\mathbf{x})(\epsilon)=f(\mathbf{y})(\epsilon)$ for any $\epsilon\in \{0,1\}^d\backslash\{\vec{0}\}$
and $(f(\mathbf{x})(\vec{0}),f(\mathbf{y})(\vec{0}))\in \RP^{[\infty]}(X)$.
Moreover, $f(\mathbf{y})$ is also a $T^{[d]}$-minimal point.
By Case 1, we get that $f(\mathbf{y})\in \Q^{[d]}(X)$
and thus $\mathbf{y}\in \Q^{[d]}(X)$.
\end{proof}

Recall a characterization of $\I_{fip}$-pairs in \cite[Corollary 4.4]{DDMSY13}.

\begin{lemma}\label{mini-infi}
  Let $(X,T)$ be a minimal system and $(x_0,x_1)\in \RP^{[\infty]}(X)\backslash \Delta$.
  If $(x_0,x_1) $ is a $T\times T$-minimal point,
  then $(x_0,x_1)$ is an $\I_{fip}$-pair.
\end{lemma}

Analogously to Lemma \ref{mini-infi},
we provide a characterization of $\mathrm{IN}^{[d]}$-pairs.

\begin{lemma}\label{main-thm3}
Let $(X,T)$ be a minimal system and $d\in \N,(x_0,x_1)\in \mathbf{RP}^{[d]}(X)\backslash \Delta$.
 If $(x_0,x_1)$ is a $T\times T$-minimal point, then
$(x_0,x_1)$ is an $\mathrm{IN}^{[d]}$-pair.
\end{lemma}

\begin{proof}
Let $\pi:X\to X_\infty=X/\mathbf{RP}^{[\infty]}(X)$ be the factor map
and let $u_j=\pi(x_j),j=0,1$.
If $u_0=u_1$, then $(x_0,x_1)\in \RP^{[\infty]}(X)$ and thus $(x_0,x_1)\in \I_{fip}(X)$
by Lemma \ref{mini-infi}.
In particular, we have $(x_0,x_1)\in \mathrm{IN}^{[d]}(X)$.

Now assume that $u_0\neq u_1$,
then $(u_0,u_1)\in \RP^{[d]}(X_\infty)\backslash\Delta_{X_\infty}$ by Theorem \ref{lift-property}
and $(u_0,u_1)$ is a $T\times T$-minimal point as
$(x_0,x_1)$ is a $T\times T$-minimal point.
\medskip


Fix $k\in \N$.
By Lemma \ref{INN}, it suffices to show that
there is some $\mathbf{x}\in\Q^{[k(2^d+d)]}(X)$ such that
 $\mathbf{x}(\theta)=x_{t_a(b)}$
for all $\theta\in \Theta_{k,d}$.

\medskip

\noindent {\bf Step 1: Reduction on the maximal factor of order $\infty$.}

It follows from Theorem \ref{system-of-order} that
$X_\infty$ is an inverse limit of minimal nilsystems.
By the argument in Corollary \ref{inverse-limit},
there exists some $\mathbf{u}\in\Q^{[k(2^d+d)]}(X_\infty)$ such that
 $\mathbf{u}(\theta)=u_{t_a(b)}$
for all $\theta\in \Theta_{k,d}$.

\medskip

\noindent {\bf Step 2: Lifting to $X$.}

Notice that $\pi^{[l]}:(\Q^{[l]}(X),\mathcal{G}^{[l]})\to (\Q^{[l]}(X_\infty),\mathcal{G}^{[l]})$ is a factor map for every $l\in \N$,
where $\pi^{[l]}: X^{[l]}\rightarrow X_\infty^{[l]}$
is defined from $\pi $ coordinatewise.

We note that Theorem \ref{lift-property} also holds for general abelian groups action,
thus there is some $\mathbf{w}\in \mathbf{Q}^{[k(2^d+d)]}(X)$
such that
\[
\pi^{[k(2^d+d)]}(\mathbf{w})=\mathbf{u},
\]
which implies that $\mathbf{w}(\theta)\in \RP^{[\infty]}[x_{t_a(b)}] $
for all $\theta\in \Theta_{k,d}$.

\medskip

\noindent {\bf Step 3: Transformations.}

\medskip

\noindent {\bf Case 1:} $(x_0,x_1,\mathbf{w})$ is a $T^{[k(2^d+d)]+2}$-minimal point.

By Lemma \ref{replace},
we can replace $\mathbf{w}(\theta)$ by $x_{t_a(b)}$ for all $\theta\in \Theta_{k,d}$
which implies that there is some $\mathbf{x}\in \mathbf{Q}^{[k(2^d+d)]}(X)$ such that
 $\mathbf{x}(\theta)=x_{t_a(b)}$
for all $\theta\in \Theta_{k,d}$.

\medskip

\noindent {\bf Case 2: General cases.}

By property (3) of Proposition \ref{ELLIS}, there is a minimal point
\begin{equation}\label{proximal}
(x_0',x_1',\mathbf{w}')
\in \overline{\mathcal{O}((x_0,x_1,\mathbf{w}),T^{[k(2^d+d)]+2})}
\end{equation}
such that they are also proximal.
Note that $\mathbf{Q}^{[k(2^d+d)]}(X)$ is $T^{[k(2^d+d)]}$-invariant,
we get $\mathbf{w}'\in \mathbf{Q}^{[k(2^d+d)]}(X)$.

Now by (\ref{proximal}),
$(x_i,x_i'),(\mathbf{w}(\epsilon),\mathbf{w}'(\epsilon))\in\mathbf{P}(X) $
for all $i=0,1$ and $\epsilon\in \{0,1\}^{[k(2^d+d)]}$.
As $\mathbf{P}(X)\subset \mathbf{RP}^{[\infty]}(X)$
and $\mathbf{w}(\theta)\in \RP^{[\infty]}[x_{t_a(b)}]$ for all $\theta\in \Theta_{k,d}$,
 by equivalence of $\RP^{[\infty]}(X)$
we get that
\[
\mathbf{w}'(\theta)\in \RP^{[\infty]}[x'_{t_a(b)}],
\]
which implies $\mathbf{w}'\in \Q^{[k(2^d+d)]}(X)$  by Case 1.

Recall that
 $(x_0,x_1)$ is a $T\times T$-minimal point and
  $ (x_0',x_1')\in \overline{\mathcal{O}((x_0,x_1),T\times T)}$,
 there exists some sequence $\{n_i\}_{i\in \N}\subset \Z$ such that
 \[
 (T\times T)^{n_i} (x_0',x_1')\to (x_0,x_1)
 \]
 as $i\to \infty$.
 Let $\mathbf{x}$ be some limit point of the sequence $\{ (T^{n_i})^{[k(2^d+d)]} \mathbf{w}'\}_{i\in \N}$,
 then we have $\mathbf{x}\in \Q^{[k(2^d+d)]}(X)$ and $\mathbf{x}(\theta)=x_{t_a(b)}$ for all $\theta\in \Theta_{k,d}$.
\end{proof}

As a consequence,
we get the following corollary.

\begin{cor}\label{distal-case}
 Let $(X,T)$ be a minimal distal system and $d\in \N$.
  Then $(x_0,x_1)\in\mathrm{IN}^{[d]}(X)$
if and only if $(x_0,x_1)\in \RP^{[d]}(X)$.
\end{cor}

Now we are able to show the main result of this section.
We need the following theorem.

\begin{theorem}\cite[Theorem 4.5]{DDMSY13}\label{infi-strp}
Let $(X,T)$ be a minimal system. If $X$ does not contain any nontrivial
$\mathrm{Ind}_{fip}$-pair, then it is an almost one-to-one extension of
its maximal factor of order $\infty$.
\end{theorem}

\begin{theorem}\label{main-thm2}
Let $(X,T)$ be a minimal system and $d\in \N$.
If $X$ does not contain any nontrivial $\mathrm{IN}^{[d]}$-pair,
then it is an almost one-to-one extension of
its maximal factor of order $d$.
\end{theorem}

\begin{proof}
Let $(X,T)$ be a minimal system without
$\mathrm{IN}^{[d]}$-pairs, where $d\in \N$.
Let $\pi:X\to X/\mathbf{RP}^{[d]}(X)$ be the factor map.

We first show that $\pi$ is a proximal extension.

Remark that if $(x,y)\in R_{\pi}=\mathbf{RP}^{[d]}(X)$ is a $T\times T$-minimal point,
then by Lemma \ref{main-thm3} we have $(x,y)$ is an $\mathrm{IN}^{[d]}$-pair and thus we get that $x=y$.
Now consider any $(x,y)\in R_{\pi}$ and $u\in E(X,T)$ a minimal idempotent.
As $(ux,uy)$ is a $T\times T$-minimal point,
we have from previous observation that $ux=uy$, which implies that $(x,y)$ is a proximal pair.

This shows
 that $\mathbf{P}(X)=\mathbf{RP}^{[\infty]}(X)=\mathbf{RP}^{[d]}(X)$,
 which implies that the maximal factor of order $\infty$ of $X$
 is $X/\mathbf{RP}^{[d]}(X)$.

As $X$ dose not contain any nontrivial $\mathrm{IN}^{[d]}$-pair,
we get that $\mathrm{Ind}_{fip}(X)$ is trivial.
By Theorem \ref{infi-strp} $X$ is an almost one to one extension of
its maximal factor of order $\infty$.
From this, we deduce that
$X$ is an almost one to one extension of
its maximal factor of order $d$.

This completes the proof.
\end{proof}

To end this section,
we give a question which we can not solve in this paper:

\begin{ques}
Let $(X,T)$ be a minimal system without any nontrivial $\mathrm{IN}^{[d]}$-pair,
where $d\in \mathbb{N}$.
Is $(X,T)$ uniquely ergodic?
\end{ques}

\appendix

\section{Basic facts about abstract topological dynamics}\label{sectionA}

In this appendix, we recall some basic definitions and results in abstract topological dynamical systems.
For more details, see \cite{JA88,ER}.

\subsection{Topological transformation groups}

A \emph{topological dynamical system} is a triple $\mathcal{X}=(X,\mathcal{T},\Pi)$,
where $X$ is a compact metrizable space, $\mathcal{T}$ is a $T_2$ topological group and $\Pi:T\times X\to X$
is a continuous map such that $\Pi(e,x)=x$
and $\Pi(s,\Pi(t,x))=\Pi(st,x)$.
We shall fix $\mathcal{T}$ and suppress the action symbol.
In lots of literatures, $\mathcal{X}$ is also called a \emph{topological transformation group} or a \emph{flow}.
Usually we omit $\Pi$ and denote a system by $(X,\mathcal{T})$.

Let $(X,\mathcal{T})$ be a system and $x\in X$, then $\mathcal{O}(x,\mathcal{T})$
denotes the \emph{orbit} of $x$, which is also denoted by $\mathcal{T}x$.
A subset $A\subset X$ is called \emph{invariant} if $ta\in A$ for all $a\in A$ and $t\in \mathcal{T}$.
When $Y\subset X$ is a closed and $\mathcal{T}$-invariant subset of the system $(X,\mathcal{T})$
we say that the system $(Y,\mathcal{T})$ is a \emph{subsystem} of $(X,\mathcal{T})$.
If $(X,\mathcal{T})$ and $(Y,\mathcal{T})$ are two dynamical systems their \emph{product system} is the system
$(X\times Y,\mathcal{T})$, where $t(x,y)=(tx,ty)$.
A system $(X,\mathcal{T})$ is called \emph{minimal} if $X$ contains no proper closed invariant subsets.

\subsection{Enveloping semigroups}
Given a system $(X,\mathcal{T})$ its \emph{enveloping semigroup}
or \emph{Ellis semigroup} $E(X,\mathcal{T})$ is defined as the closure of the set $\{t:t\in\mathcal{T}\}$ in $X^X$
(with its compact, usually non-metrizable, pointwise convergence topology).
The maps
$E\to E:p \mapsto  pq$ and $p\mapsto  tp$ are continuous for all $q\in E$ and $t\in \mathcal{T}$.

\subsection{Idempotents and ideals}
For a semigroup the element $u$ with $u^2 =u$ is called
an \emph{idempotent}. Ellis-Namakura Theorem says that for any enveloping semigroup $E$
the set $J(E)$ of idempotents of $E$ is not empty \cite{ER}.
A non-empty subset $I\subset E$ is
a \emph{left ideal} (resp. \emph{right ideal}) if it $EI\subset I$ (resp. $IE\subset I$). A \emph{minimal left ideal} is
the left ideal that does not contain any proper left ideal of $E$.
Obviously every left
ideal is a semigroup and every left ideal contains some minimal left ideal.

An idempotent $u \in J(E)$ is \emph{minimal} if $v \in J(E)$ and $vu = v$ implies $uv = u$. The following
results are well-known \cite{EEN01,FK}: let $ L$ be a left ideal of enveloping semigroup $E$ and $u \in J(E)$.
Then there is some idempotent $v$ in $Lu$ such that $uv = v$ and $vu = v$; an idempotent is minimal
if and only if it is contained in some minimal left ideal.

A useful result about minimal point is as follows:

\begin{prop}
  Let $I$ be a minimal left ideal.
  A point $x\in X$ is minimal if and only if $ux=x$ for some $u\in I$.
\end{prop}
\subsection{Proximality}
Two points $x_1$ and $x_2$ are called \emph{proximal} if and only if
\[
\overline{\mathcal{T}(x_1,x_2)}\cap \Delta_X\neq \emptyset.
\]
Let $\mathcal{U}_X$ be the unique uniform structure of $X$, then
\[
\mathbf{ P}(X ) =\bigcap \{\mathcal{T}\alpha:\alpha\in \mathcal{U}_X \}
\]
is the collection of proximal pairs in $X$, the \emph{proximal relation}.
\begin{prop}\label{ELLIS}
  Let $(X,\mathcal{T})$ be a dynamical system. Then
  \begin{enumerate}
    \item The points $x_1,x_2$ are proximal in $(X,\mathcal{T})$ if and only if $px_1=px_1$ for some $p\in E(X,\mathcal{T})$.
    \item If $u$ is an idempotent in $E(X,\mathcal{T})$, then $(x,ux)\in \mathbf{P}(X)$ for every $x\in X$.
    \item There is a minimal point $x'\in  \overline{\mathcal{O}(x,\mathcal{T} ) }$ such that
     $(x,x')\in \mathbf{P}(X)$.
    \item  If $(X,T)$ is minimal, then $(x, y) \in \mathbf{P}(X)$ if and only if there is some minimal idempotent
    $u\in E(X,\mathcal{T})$ such that $y = ux$.
  \end{enumerate}
\end{prop}

\end{document}